\documentclass[11pt]{amsart}
\usepackage{graphicx}

\theoremstyle{theorem}
\newtheorem{theorem}{Theorem}

\newcommand{\N}{\mathbf{N}}
\DeclareMathOperator{\E}{\mathbf{E}}
\DeclareMathOperator{\p}{\mathbf{P}}
\DeclareMathOperator{\I}{\mathbf{1}}
\newcommand{\eps}{\epsilon}

\begin{document}

\title{A Gambler that Bets Forever and the Strong Law of Large Numbers}
\markright{Gambling and the Law of Large Numbers}
\author{Calvin Wooyoung Chin}
\date{\today}

\begin{abstract}
In this expository note, we give a simple proof that a gambler repeating a game
with positive expected value never goes broke with a positive probability.
This does not immediately follow from the strong law of large numbers or
other basic facts on random walks.
Using this result, we provide an elementary proof of
the strong law of large numbers.
The ideas of the proofs come from the maximal ergodic theorem
and Birkhoff's ergodic theorem.
\end{abstract}

\maketitle

A gambler that makes the same bet on a single number over and over
in a game of American roulette is bound to go broke with probability $1$.
This has to do with the fact that the house has an edge of 5.26\% in average,
and it turns out that the same is true for any other game with negative
expected value.
This general fact, sometimes called \emph{gambler's ruin},
is a direct consequence of the strong law of large numbers.
To see this, let $X_1,X_2,\ldots$ be i.i.d.\ (real-valued) random variables
with finite mean, representing the gains from the games.
For each $n \in \N$, write $S_n := X_1+\cdots+X_n$, denoting
the capital of the gambler after the $n$-th game.
The strong law of large numbers says that $S_n/n \to \E X_1$ with
probability $1$. If $\E X_1 < 0$, then $S_n/n$ takes a negative value
at some point with probability $1$, and thus so does $S_n$.
This proves gambler's ruin.

What if the game favors the gambler in the sense that $\E X_1 > 0$?
Is it possible that the game continues forever without the gambler going broke?
Mathematically, do we have
\begin{equation} \label{eq:pos_prob}
\p(S_n > 0 \text{ for all $n\in\N$}) > 0?
\end{equation}
An argument similar to the one given above shows that $S_n$ stays positive
after some point with probability $1$.
However, this does not exclude the possibility of having $S_n < 0$
before $S_n$ eventually remains positive.

If $X_1$ only takes $\pm 1$ as values, a classical approach
\cite[Section 6.2]{Ash70}
involving difference equations yields \eqref{eq:pos_prob}.
More generally, if $X_1$ is bounded and $\p(X_1<0)>0$, then we can use the fact
that $(e^{r S_n})_{n\in\N}$ is a martingale for some $r < 0$ to
obtain~\eqref{eq:pos_prob}; see \cite[Theorem 4.8.9]{Dur19}.
However, these methods do not easily generalize to arbitrary $X_1$
with finite mean.

In this expository note, we first illustrate a proof of
\eqref{eq:pos_prob} for arbitrary $X_1$ with positive mean.
The idea behind the proof is to observe that we need to show
the so-called \emph{maximal ergodic theorem} \cite[Lemma 6.2.2]{Dur19}.

We then use our result to prove the strong law of large numbers.
The idea comes from the proof of \emph{Birkhoff's ergodic theorem}
\cite[Theorem 6.2.1]{Dur19}, but our proof is rather elementary.
In particular, we do not explicitly invoke advanced theorems, and
we avoid using notions such as limit superior.
Given that we obtain an interesting theorem on gambling along the way,
our exposition offers a clear and efficient approach to the strong law
of large numbers.

Let us start by proving \eqref{eq:pos_prob} for arbitrary $X_1$
with positive mean.

\begin{theorem} \label{thm:eternal}
If $\E[X_1] > 0$, then $\p(S_n>0\text{ for all $n\in\N$}) > 0$.
\end{theorem}

\begin{proof}
Let us write $a_1\wedge\cdots\wedge a_n$ to denote the minimum
among $a_1,\ldots,a_n$.
For each $n \in \N$, we claim the following:
\begin{equation} \label{eq:birkhoff_claim}
0 \wedge S_1 \wedge \cdots \wedge S_n \ge
X_1\I_{\{S_1\wedge\cdots\wedge S_n \le 0\}}
+ 0 \wedge X_2 \wedge \cdots \wedge (X_2+\cdots+X_n).
\end{equation}
If $S_1\wedge\cdots\wedge S_n\le 0$, then the equality holds;
see Figure \ref{fig}.
If $S_1\wedge\cdots\wedge S_n > 0$, then the left side is $0$,
while the right side is nonpositive. Thus, \eqref{eq:birkhoff_claim} is proved.
\begin{figure}[h]
\includegraphics[width=\textwidth]{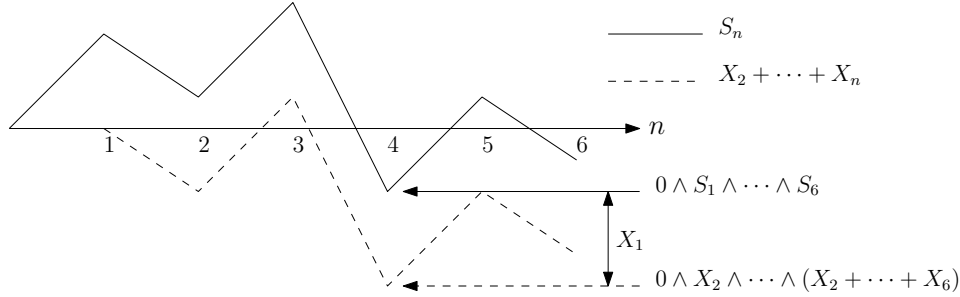}
\caption{Illustration of \eqref{eq:birkhoff_claim} when
$S_1\wedge\cdots\wedge S_n \le 0$.}
\label{fig}
\end{figure}

Since
\[
\begin{split}
\E[0\wedge X_2\wedge\cdots\wedge (X_2+\cdots+X_n)]
&= \E[0\wedge S_1\wedge\cdots\wedge S_{n-1}] \\
&\ge \E[0\wedge S_1\wedge\cdots\wedge S_n],
\end{split}
\]
taking the expectation of both sides of \eqref{eq:birkhoff_claim} yields
\[ \E[X_1\I_{\{S_1\wedge\cdots\wedge S_n \le 0\}}] \le 0. \]
Letting $n \to \infty$, we have
\[ \E[X_1\I_{\{S_n\le0\text{ for some }n\in\N\}}] \le 0 \]
(by dominated convergence).
Since $\E[X_1]>0$, we have
\[ \p(S_n\le0\text{ for some }n\in\N) < 1, \]
which is equivalent to the desired conclusion.
\end{proof}

We continue on to prove the strong law of large numbers.

\begin{theorem}[strong law of large numbers]
We have $S_n/n \to \E[X_1]$ with probability $1$.
\end{theorem}

\begin{proof}
We may assume that $\E[X_1] = 0$.
Given $\eps > 0$, let $X_n+\eps$ play the role of $X_n$ in
Theorem~\ref{thm:eternal}. Then we have
\[ \p(S_n+n\eps>0 \text{ for all $n\in\N$}) > 0, \]
which is equivalent to
\[ \p(S_n/n > -\eps \text{ for all $n\in\N$}) > 0. \]
Let $p$ denote the positive probability on the left side.
If we write ``eventually" to mean ``for all $n\ge M$ for some $M\in\N$,"
then for any $N \in \N$ we have
\[ \{S_n/n > -\eps \text{ eventually}\} =
\{(X_N+\cdots+X_n)/n > -\eps \text{ eventually}\}. \]
So, $\{S_n/n>-\eps\text{ for all }n<N\}$ and $\{S_n/n>-\eps \text{ eventually}\}$
are independent, and thus
\[
\p(S_n/n > -\eps \text{ for all }n<N)
\p(S_n/n > -\eps \text{ eventually}) \ge p.
\]
Letting $N\to\infty$ yields
\[ p\p(S_n/n>-\eps \text{ eventually}) \ge p. \]
Since $p > 0$, we have
\[ \p(S_n/n>-\eps \text{ eventually}) = 1. \]
The symmetrical argument yields
\[ \p(S_n/n < \eps \text{ eventually}) = 1. \]
Let $A_\eps$ and $B_\eps$ be the events in the previous two displays.
On $\bigcap_{k=1}^\infty (A_{1/k}\cap B_{1/k})$, we have $S_n/n \to 0$.
Since
\[ \p\biggl(\bigcap_{k=1}^\infty (A_{1/k}\cap B_{1/k})\biggr) = 1, \]
the proof is completed.
\end{proof}



\begin{thebibliography}{2}

\bibitem{Ash70} Ash, R. B. (1970). \textit{Basic Probability Theory.} New York--London--Sydney: John Wiley \& Sons, Inc.

\bibitem{Dur19} Durrett, R. (2019). \textit{Probability: Theory and Examples,} 5th ed. Cambridge Series in Statistical and Probabilistic Mathematics, 49. Cambridge: Cambridge University Press.

\end{thebibliography}
\end{document}